\documentclass[reqno]{amsart}
\usepackage{amsmath,amssymb,cite}
\usepackage[mathscr]{euscript}
\usepackage{mathtools}
\usepackage{xcolor}
\usepackage{stix}
\usepackage{hyperref}
\usepackage[final]{showkeys}
\usepackage{dsfont}
\usepackage{relsize}
\usepackage[normalem]{ulem}
\usepackage{bbold,dsfont}

\usepackage[final]{fixme}

\newcommand{\overbar}[1]{\mkern 2.0mu\overline{\mkern-4mu#1\mkern-2.0mu}\mkern 2.0mu}
\newcommand{\overbarr}[1]{\mkern 2.0mu\overline{\mkern-3mu#1\mkern-2.0mu}\mkern 3.0mu}

\definecolor{bblue}{rgb}{.2,0.2,.8}

\theoremstyle{plain}
\newtheorem{theorem}{Theorem}[section]
\newtheorem{proposition}[theorem]{Proposition}
\newtheorem{lemma}[theorem]{Lemma}
\newtheorem{corollary}[theorem]{Corollary}

\theoremstyle{definition}
\newtheorem{definition}[theorem]{Definition}

\theoremstyle{remark}
\newtheorem{remark}[theorem]{Remark}

\newcommand{\T}{\overbarr{T}}
\DeclareMathOperator{\tdeg}{tdeg}

\numberwithin{equation}{section}
\numberwithin{theorem}{section}

\def\be{\begin{equation}}
	\def\ee{\end{equation}}
\def\bp{\begin{pmatrix}}
	\def\ep{\end{pmatrix}}
\def\bea{\begin{eqnarray}}
	\def\eea{\end{eqnarray}}

\def\\{\par\medskip}

\let\0=\noindent

\newcommand*{\defeq}{\mathrel{\vcenter{\baselineskip0.5ex \lineskiplimit0pt
			\hbox{\scriptsize.}\hbox{\scriptsize.}}}%
	=}

\renewcommand{\epsilon}{\varepsilon}
\renewcommand{\phi}{\varphi}
\renewcommand{\hat}{\widehat}

\newcommand{\Z}{\mathbb{Z}}
\newcommand{\Q}{\mathbb{Q}}
\newcommand{\C}{\mathbb{C}}
\newcommand{\N}{\mathrm{N}}
\newcommand{\Lie}{\mathfrak{L}}

\newcommand{\der}{\partial}

\newcommand{\Set}[1]{\left\{#1\right\}}

\DeclareMathOperator{\wdd}{wd}
\DeclareMathOperator{\wt}{wt}
\DeclareMathOperator{\lev}{lev}
\DeclareMathOperator{\AGL}{AGL}
\DeclareMathOperator{\frk}{frk}
\DeclareMathOperator{\M}{M}
\DeclareMathOperator{\W}{W}
\DeclareMathOperator{\Sym}{Sym}
\DeclareMathOperator{\re}{Re}

\DeclareMathOperator{\LP}{\textit{VP}}

\title{Transfinite hypercentral iterated wreath product of integral domains}

\author{Riccardo Aragona}
\address{\noindent Riccardo Aragona \hfill\break\indent 
	DISIM, Universit\`a dell'Aquila
	\hfill\break\indent 
	67100 Coppito, L'Aquila, Italy
}
\email{riccardo.aragona@univaq.it}

\author{Norberto Gavioli}
\address{\noindent Norberto Gavioli \hfill\break\indent 
	DISIM, Universit\`a dell'Aquila
	\hfill\break\indent 
	67100 Coppito, L'Aquila, Italy
}
\email{norberto.gavioli@univaq.it}

\author{Giuseppe Nozzi}
\address{\noindent Giuseppe Nozzi \hfill\break\indent 
	DISIM, Universit\`a dell'Aquila
	\hfill\break\indent 
	67100 Coppito, L'Aquila, Italy
}
\email{giuseppe.nozzi@graduate.univaq.it}

\begin{document}
	
\thanks{All the authors are members of INdAM-GNSAGA (Italy); the first
  and  the second  authors are  members  of the  center of  excellence
  ExEmerge of the University of L'Aquila. The authors thankfully acknowledge support by MUR-Italy via PRIN 2022RFAZCJ "Algebraic methods in Cryptanalysis".}
\begin{abstract}
  Starting  with an  integral  domain $D$  of  characteristic $0$,  we consider a class of iterated  wreath  product $W_n$  of  $n$  copies  of
  $D$. In order that  $W_n$  be  transfinite hypercentral, it is  necessary to
  restrict to the case of wreath products defined by way of numerical
  polynomials. We also associate to each of  these groups a Lie ring, providing a correspondence preserving most of the structure. This construction generalizes a result of \cite{netreba} which characterizes the Lie algebras associated to the Sylow \(p\)-subgroups of the symmetric group \(\Sym(p^n)\).  
  As an application,  we explore the
  normalizer    chain   $\lbrace\mathbf{N}_{i}\rbrace_{i\geq    -1}$
  starting from the  canonical regular abelian subgroup  $T$ of $W_n$. Finally,
  we   characterize   the   regular  abelian   normal   subgroups   of
  $\mathbf{N}_0$ that are isomorphic to  $D^n$.
\end{abstract}
	
\noindent

\maketitle
\thispagestyle{empty}

\section{Introduction}
In a  recent paper,  Aragona et al.~\cite{char0algebra}  introduce the
 \emph{integral  ring of  partitions}  $\Lie_n$ defined as the  iterated
wreath product $\wr^{n}{\Z}$  of \(n\) copies of the  trivial Lie ring
\(\Z\), seen as a subring of the Witt Lie algebra. After defining a  chain of
idealizers  starting from  the  canonical regular  abelian subring  of
$\wr^{n}{\Z}$, they establish a relationship between the \emph{growth}
of    this   chain    and   integer    partitions   with    parts   in
$\lbrace  1,\dots,n-1\rbrace$.  Inspired by that paper, in  this
work we consider  the iterated wreath product of \(n\)  copies of the
additive group  of an integral  domain of characteristic  $0$. 
In particular, let $D$ be an integral domain of
characteristic  $0$. For $n\geq 2$ and \(1\le i \le n-1\), we consider    a special class of translation-invariant  additive
subgroup $R_{i}$ of $D^{D^{i}}$. We define the group $W_n$ recursively as the
wreath product
\[
  D\wr_{R_{n-1}} W_{n-1} \defeq R_{n-1} \rtimes W_{n-1}. \ 
\]
We start focusing on the upper central series of $W_n$, showing that a
necessary  and  sufficient  condition  for  $W_n$  to  be  transfinite
hypercentral, is that  $R_i$ is contained in a special subring \(\LP_i\) of \(D^{D^i}\), which we shall call \emph{ring of virtual polynomials} (see Definition~ \ref{def:virtual_poly} below). When \(D=\Z\), the ring \(\LP_i\)  coincides with the ring  $P_i$  of \emph{numerical polynomials}, i.e., polynomials in $i$ indeterminates with coefficients in  the fraction field \(F\) of \(D\), taking values in \(D\) when  evaluated over  \(D\), see also~\cite{polnum}.
 In  particular, when $R_i \otimes_D F =  F[x_1,\dots, x_i]$ 
for
$i   =   1,   \dots,   n-1$, we   prove   that   \(W_n\)   is   the 
\((\omega^{n-1}+\dots+\omega+1)\)-th term  of  its upper central series, where \(\omega\) is the first limit ordinal.
 The nilpotency of wreath products and their central series is a classical topic of research and has been addressed in many papers, e.g., \cite{baumslag, meldrum, neumann,polnum}.  

Next,    we    restrict    our    attention   to    the    case    when \(R_i=D[x_1,\dots, x_{i}]\).  We find a map $\phi$ that relates the group $W_n$ and the basis of the Lie ring $\Lie_n$ over \(D\). Using this correspondence, we prove that this ring is also transfinite hypercentral with the same transfinite hypercentral class. We also generalize the concept of saturated subgroups given in \cite{LieRigid} finding a natural counterpart to  homogeneous subrings as introduced in \cite{char0algebra}. 
The structure of the lattices of the saturated normal subgroups of \(W_n\) and of the homogeneous ideals of \(\Lie_{n}\) are related by the map \(\phi\). We actually prove that  \(W_n\) and \(\Lie_{n}\) are\fxnote{uno solo: o \emph{actually} o  \emph{indeed}} uniserial, i.e., every saturated normal subgroup, respectively homogeneous ideal, is a term of the transfinite upper central series. 

In \cite{char0algebra} the chain of idealizers originating from the canonical regular abelian subalgebra is studied providing a description of the growth of its factors. We observe that the map $\phi$ sends normalizers to idealizers when restricted to saturated subgroups. As a consequence we find a complete analogous result for the corresponding chain of normalizers \(\Set{\mathbf{N}_i}\) in the group \(W_n\). In particular the growth of this chain is related to the integer partitions\fxnote{\emph{integer partitions} o \emph{partition of integers}?} with parts of size at most \(n-1\). The study of iterated normalizers  of regular subgroups of the symmetric groups,  i.e., their multiple holomorphs, has been studied since a long time ago. A paper of Miller has been published in 1908 \cite{miller}, and this topic is still a subject of recent research, see e.g., \cite{jgt24,jpaa25}.

We  conclude  this  work  by  giving a  canonical  parametric  set  of
generators  of  the  regular  abelian  subgroups  of  $\mathbf{N}_0$, proving that  this family of groups can be parameterized by
way of the  elements of \(D\). These groups actually  form one or two conjugacy classes under the action of \(\mathbf{N}_1\),
according to the invertibility of  \(2\) in \(D\). This
result      is     a      \(0\)-characteristic     counterpart      of
\cite[Theorem~6]{regular}. The  study of  pairwise normalizing  regular subgroups  of the
symmetric  group has  also been  carried out  using  the
structure   of   bi-braces, e.g., see  \cite{caranti}.

\section{Preliminaries}
\subsection{Notation}
We will denote by $D$ an  integral domain of characteristic $0$ and by
$F$ the fraction field of
$D$.
When  \(\alpha\)  is  an  ordinal,  $Z_{\alpha}(G)$  will  denote  the
$\alpha$-th term of the transfinite upper central series of the group $G$. We will consistently 
use the letter \(\omega\) to denote the first limit ordinal.
	
To avoid confusion, we warn the  reader that additive notation is used
when operations  occur within  the same base  abelian subgroup  of the
wreath  product,   and  multiplicative  notation  is   used  when  the
operations involve elements of different base subgroups.
\subsection{Wreath products}
We start by recalling the general construction of wreath product first introduced by Krasner and Kaloujnine in~\cite{kal1,kal2,kal3}.
Let $K$ and $H$  be two groups, and let \(H\) act on a set $\Gamma$. 
The  unrestricted wreath product of  \(K\) and \(H\) is defined as
\begin{equation*}
  K\wr H\defeq K^\Gamma\rtimes H ,
\end{equation*}
where  $K^\Gamma$  is the  group  of  function  from $\Gamma$  to  $K$, 
equipped    with    the    point-wise   product    operation. That is, for all \(\gamma\in \Gamma\), 
\(\left(f
  g\right)\left(\gamma\right)=f\left(\gamma\right)g\left(\gamma\right).\)
 The group $H$ acts by translation on ${K^\Gamma}$ as follows
\begin{equation*}
  f^h\left(\gamma\right)\defeq f(\gamma^{h^{-1}}), \quad \gamma\in \Gamma,\ h\in H. 
\end{equation*}
The product on $K\wr H$ is defined by
\begin{equation*}
  \left(f_1,h_1\right)\cdot\left(f_2,h_2\right)=\left(f_1 f_2^{h_1^{-1}},h_1 h_2\right).
\end{equation*}
\begin{definition}\label{wreathristretto}
	Let \(K,H\) be two groups. If $L\le {K^H}$ is an $H$-invariant subgroup, we define
	\begin{equation*}
		K\wr_{L}H\defeq L\rtimes H.
	\end{equation*}
	When  \(L=K^H\) we obtain the classical notion of unrestricted standard wreath product and, in this case, we  will  omit the  subscript \(L\). The base  subgroup of \(K\wr_L  H\) is  as usual defined by $B=\lbrace \left(f,1_H\right)\mid f\in L\rbrace\cong L$.
\end{definition}
Let \(D\) be an integral  domain of characteristic \(0\) with fraction
field  \(F\).   We  shall  denote  by  \(R_i\) an  additive subgroup  of  \(D^{D^i}\) containing a copy of \(D\) seen as the subring of constant functions. We ask furthermore that \(D^i\) acts faithfully on \(R_i\) by translations and that \(f(g_0,\dots , g_{i-1})\in R_i\) whenever \(f\in R_i\) and \(g_j\in R_{j}\). 

In  this  work we  will
consider the multiple wreath product
\begin{equation*}
  W_n= D\wr_{R_{n-1}}D\wr_{R_{n-2}} \dots \wr_{R_1}D
\end{equation*}
of \(n\) copies of \(D\).  The  group \(W_n\) is a permutation group on
the set $D^n$.
	
Let \(f\in R_{k-1}\), we shall denote by $f\Delta_k$ an element in the
base group $B_k$ of
\[W_k=D\wr_{R_{k-1}} W_{k-1}\] corresponding  to the function \(f\)
when regarded  as an element of  \(B_k\). The elements of  \(W_n\) are
then   \(n\)-tuples   \((f_{n-1}\Delta_n,\dots,  f_0\Delta_1)\)   with
\(f_i\in  R_i\).  It is  worth  noting  that  \(R_0\) is  an  additive
subgroup of \(D\).   The action of \(W_n\) on \(D^{n}\)  is defined as
follows
	\begin{multline*}
		\left(x_1,x_2,\ldots,x_n\right)\cdot \left(f_{n-1}\Delta_n,\ldots,f_1\Delta_2,f_0\Delta_1\right)=\\
		\left(x_1-f_0,x_2-f_{1}\left(x_1\right),                     \ldots,
		x_n-f_{n-1}\left(x_1,\ldots,x_{n-1}\right)\right).
\end{multline*}

Every base subgroup \(B_k\) is endowed with a structure of \(D\)-module. The action of \(D\) on \(B_k\) is given by the multiplication on the value of a function. More in details, we have \(d\cdot (f\Delta_k)=(df)\Delta_k\), where \(d\in D\). We extend this multiplication action to the whole group by letting \(d\prod_{k=n}^1 f_k\Delta_k=\prod_{k=n}^1 df_k\Delta_k\).

\begin{definition}
		Let $G\le W_n$, we shall say that $G$ is a $D$-subgroup of $W_n$ if 
		\(G\) is closed under multiplication by elements of \(D\), i.e.,
		\[G=\Set{dg \mid d\in D, g\in G}=D G.\]
	\end{definition}
	All the subgroups of $W_n$ considered in this paper will be $D$-subgroups, unless otherwise stated.
\subsection{Transfinite hypercentral groups}
Let $G$ be a group, we shall denote by \(Z_i(G)\) the \(i\)-th term of
the     upper     central     series      of     \(G\),     and     by
\(Z_{\omega}(G)=\bigcup_{i=1}^\infty   Z_i(G)\)  the   hypercenter  of
\(G\), where  \(\omega\) is  the first  limit ordinal.  More generally
given an ordinal \(\alpha\) then
\begin{equation*}
  Z_{\alpha}(G)=\begin{cases}
    \{g\in G \mid [g,z]\in Z_{\alpha-1}(G) \text{ for all } z\in Z_{\alpha-1}(G)\} & \text{if \(\alpha\) is a non-limit ordinal} \\
    \bigcup_{\beta < \alpha}Z_{\beta}(G) & \text{otherwise} 			
  \end{cases}
\end{equation*}
We recall that a group \(G\) is said to be transfinite hypercentral if
\(G=Z_{\alpha}(G)\) for some, possibly  limit, ordinal \(\alpha\). See
also \cite[Chap. 12]{robinson} for more details.
\subsection{Difference operators}\label{subDelta}
Let  \(x\in D^j\)  and \(e_i\)  be the \(i\)-th  vector of  the canonical
basis  of  \(D^j\).   For  each  $h\in  D$,  we  define  the  operator
$\Delta_i(h)\colon          R_j          \to          R_j$          by
\begin{equation}\label{eq:deltaoperator}
\Delta_i(h)f  (x)= f(x+he_i)-f(x).
\end{equation}
It is worth to highlight that we denote by \(f\Delta_k\) an element in the \(k\)-th base subgroup \(B_k\) of \(W_n\), and  with a slight abuse of notation, we will also use the same symbol \(\Delta\) to denote the operator defined in Equation~\eqref{eq:deltaoperator}. It allow us to conveniently express the commutator between two elements \(f\Delta_k\) and \(g\Delta_u\) of the group. In particular, if \(k>u\) we have
\begin{align*}
	[f\Delta_k,g\Delta_u]&=(f\Delta_k)^{-1}(g\Delta_u)^{-1}(f\Delta_k)(g\Delta_u)\\
	&=(-f\Delta_k)(-g\Delta_u)(f\Delta_k)(g\Delta_u)\\
	&=(-f+f^{-g\Delta_u})\Delta_k=(\Delta_u\left(g\right)f)\Delta_k.
\end{align*}
In general, recalling that \(B_k\) is an abelian subgroup of \(W_n\), we obtain  
\begin{equation}\label{parallelism}
	\left[f\Delta_k,g\Delta_u\right]=\begin{cases}
		\left(\Delta_u\left(g\right)f\right)\Delta_k&\text{if } k>u\\
		-\left(\Delta_k\left(f\right)g\right)\Delta_u&\text{if } u>k\\
		0&\text{if } k=u
	\end{cases}
\end{equation}
Moreover if $k>u$ and \(f\) is a polynomial, by Taylor formula, we have
\begin{equation}\label{eqTaylor}
	[f\Delta_k,g\Delta_u]=\mathlarger{\sum}_{s=1}^\infty \frac{1}{s!}\frac{\der^s f}{\der x_u^s}g^s\Delta_k.
\end{equation}

	
Although the following result is well known, we sketch a proof for the convenience
of the reader
\begin{lemma}\label{lem:finite_diffs_on_Z}
  If \(f\in D^\Z\), then  $\Delta_1\left(1\right)^{k}f=0$ if and only if
  $f$ is a polynomial in \(F[x]\) of degree at most $k-1$.
\end{lemma}
\begin{proof}
  Let  $f\in F[x]$  be  a polynomial  of degree  at  most $k-1$,  then
  $f=\sum_{i=0}^{k-1} a_ix^i$. Thus
  \begin{align*}
    g(x)=\Delta_1\left(1\right)f\left(x\right)&=f\left(x+1\right)-f\left(x\right)\\
                                              &=\sum_{i=0}^{k-1} a_i\left(x+1\right)^i-\sum_{i=0}^{k-1}a_ix^i
  \end{align*}
  is  a   polynomial  of   degree  at   most  \(k-2\).   By  induction
  \(\Delta_1\left(1\right)^{k}f=\Delta_1\left(1\right)^{k-1}g=0\).
		
  Conversely let $f:D\to D$  be such that $\Delta_1(1)^k(f)=0$. Notice
  that
  \begin{equation*}
    0=\Delta_1(1)^k f(x)=\sum_{i=0}^k\binom{k}{i}\left(-1\right)^i f\left(x+k-i\right)
  \end{equation*}
  is  an   homogeneous  linear   difference  equation   with  constant
  coefficients whose solutions are
  \begin{equation*}
    f\left(x\right)=c_1+c_2 x+\ldots+c_{k-1}x^{k-1},
  \end{equation*}
  where  \(c_i\in   F\)  (for   a  detailed   proof  see~\cite[Chapter
  VI]{differences}).
\end{proof}
	
\subsection{Power monomials}\label{power monomials}
We  introduce some notation  as  in~\cite{char0algebra}.   Let
$\Lambda=\{\lambda_i\}_{i=1}^\infty$  be  a sequence  of  non-negative
integers with finite support and weight
\begin{equation}\label{defwt}
  \mathrm{wt}(\Lambda)\defeq \sum_{i=1}^\infty i\lambda_i< \infty,
\end{equation}
we   shall   say  that   $\Lambda$   is   a   partition  of   $N$   if
$\mathrm{wt}(\Lambda)=N$. The maximal part of $\Lambda$ is the integer
$\max\{i\mid  \lambda_i\neq 0\}$.   The  set of  the partitions  whose
maximal  part   is  less  than   or  equal   to  $k$  is   denoted  by
$\mathcal{P}(k)$. The power monomial  $x^\Lambda$ , where $\Lambda$ is
a partition, is defined as
\begin{equation*}
  x^\Lambda=\prod_{i=1}^\infty x_i^{\lambda_i}.
\end{equation*}
Notice    that    $x^\Lambda\Delta_k$    is    well    defined    when
$\Lambda\in\mathcal{P}(k-1)$.  We set
\begin{equation}\label{eq: base group}
	\mathcal{B}=\lbrace x^\Lambda\Delta_k\mid 1\leq k\leq n,\
  \Lambda\in \mathcal{P}(k-1)\rbrace.
\end{equation}
We call monomial elements  those of the form \(dx^\Lambda\Delta_k\), where \(d\in D\). The elements of \(\mathcal{B}\) are precisely the monic monomial elements.
	
\begin{definition}\label{def:tdeg}
  Let  $k\geq 1$  be an  integer, \(0\ne d\in D\) and  $\Lambda\in \mathcal{P}(k)$,  we
  define the  \emph{transfinite degree}  of the  monomial $dx^\Lambda$,
  written $\tdeg(x^\Lambda)$, by
  \begin{equation*}
    \tdeg(dx^\Lambda)=\omega^{n-2}\lambda_{n-1}+\dots+\omega\lambda_2+\lambda_1.
  \end{equation*}
  We set 
  \[\tdeg(dx^\Lambda\Delta_k)=         \sum_{i=1}^{n-k}\omega^{n-i}        +
  \tdeg(x^\Lambda)\] 
  where, for $k=n$, the sum has to be intended empty. In general, the transfinite degree of \(x^\Lambda\Delta_k\) does not include the term \(\omega^{k-1}\).
\end{definition}
 It is  easy  to  see  that  for  every  non-limit  ordinal  \(\alpha\),
satisfying \( 0\le \alpha \le  \sum_{i=1}^{n} \omega^{n-i}\), there exists a
unique monic   monomial  element   \(b_\alpha\in  \mathcal{B}\)   such  that
\(\tdeg b_\alpha=  \alpha\).  Notice  that this definition  endows the
set $\mathcal{B}$ with a total  ordering.  In particular, every element
$f\in W_n$ can be uniquely  decomposed as $ f=\prod_\alpha f_\alpha$,
where   $f_\alpha=c_\alpha  b_\alpha$, with \(c_\alpha\in D\), and \(b_\alpha\in \mathcal{B}\) is  the unique monic monomial   element with
transfinite degree equal to $\alpha$.   We shall denote by $\M(f)$ the
\emph{leading  term}  of  a  non-identity  element  \(f\),  i.e.   the
monomial appearing in  $f$ with non-zero coefficient  and with maximum
transfinite  degree. We define \[\tdeg(f)=
\begin{cases}
	\tdeg(\M(f)) & \text{if } f \text{ is not constant}, \\ 
	0 & \text{otherwise}.
\end{cases}
\]  We  shall write  \(f  \prec g  \) to mean that
\(\tdeg (f) < \tdeg(g)\).


\begin{lemma}\label{lem:maxmon}
  Let $x^{\Lambda}\Delta_k,x^{\Theta}\Delta_u\in W_n$ be two monomials
  such that $k>u$. The following equality holds
  \begin{equation*}
    \mathrm{M}\left([x^{\Lambda}\Delta_k,x^{\Theta}\Delta_u]\right)=\frac{\der x^{\Lambda}}{\der x_u}x^{\Theta}\Delta_k.
  \end{equation*}
\end{lemma}
\begin{proof}
  By Equation \ref{eqTaylor} we get
  \begin{equation*}
    [x^{\Lambda}\Delta_k,x^{\Theta}\Delta_u]=\sum_{s=1}^\infty \frac{\der^s x^{\Lambda}}{\der x_u^s}\frac{(x^{\Theta})^s}{s!}\Delta_k.
  \end{equation*}
    Notice  that         if          $\lambda_u=0$,         then
  $[x^{\Lambda}\Delta_k,x^\Theta\Delta_u]=0$. Otherwise, it suffices to prove that for $s\geq 1$
  \begin{equation}\nonumber
    \frac{\der^s x^{\Lambda}}{\der x_u^s}\frac{(x^{\Theta})^s}{s!}\prec \frac{\der^{s-1} x^{\Lambda}}{\der x_u^{s-1}}\frac{(x^{\Theta})^{s-1}}{(s-1)!}.
  \end{equation}
  We denote by $x^{\bar{\Lambda}}$ the monomial $x^{\Lambda}$ with the
  variable $x_u$  removed and by  $\lambda_u$ the exponent  with which
  $x_u$ appears in the monomial $x^{\Lambda}$. We get
  \begin{align*}
    \frac{\der^s x^{\Lambda}}{\der x_u^s}\frac{(x^{\Theta})^s}{s!}&=c_1x^{\bar{\Lambda}} x^{s\Theta} x_u^{\lambda_u-s}\\
                                                                  &\prec c_1 x^{\bar{\Lambda}} x^{(s-1)\Theta} x_u^{\lambda_u-(s-1)}\\
                                                                  &=c_2 \frac{\der^{s-1} x^{\Lambda}}{\der (x_u^{s-1})}\frac{(x^{\Theta})^{s-1}}{(s-1)!}.
  \end{align*}
  for suitable \(c_1,c_2\in D\). The claim follows since the \(\tdeg\) is invariant under multiplication by constants.
\end{proof}

\subsection{The integral ring of partition}\label{def:Lie Ring}
We  briefly recall  the  construction of  the   \emph{Lie $D$-ring  of
  partition} as done in~\cite{char0algebra}. This ring can be seen as an iterated wreath product of the onedimensional Lie ring \(D\). The reader is referred to \cite{MR2262857} for further details on wreath products of Lie algebras.

We  denote by  \(\partial_k\)  the derivation  given  by the  standard
partial derivative with respect to $x_k$,  where $1 \leq k \leq n$ and
we define \(\Lie_n\) to be the free \(D\)-module spanned by the basis
\begin{equation}\label{eq:base algebra}
	\mathfrak{B}:=\Set{x^\Lambda\partial_k \mid 1\le k \le n
		\text{ and } \Lambda\in \mathcal{P}(k-1)}.
\end{equation}
In the same fashion of Definition~\ref{def:tdeg} we define the transfinite degree of an element \(x^\Lambda\partial_k\) in \(\mathfrak{B}\) as  $\tdeg(x^\Lambda\partial_k)=         \sum_{i=1}^{n-k}\omega^{n-i}        +
\tdeg(x^\Lambda)$ and \(\tdeg(0)=0\). In the same way as for the group \(W_n\) the function \(\tdeg\) can be extended to \(\Lie_n\).

The product of $\Lie_{n}$
is defined on the basis \(\mathcal{B}\) via
\begin{eqnarray*}
  \left[x^\Lambda \partial_k , x^\Theta \partial_j\right] :=&
                                                              \partial_{j}(x^\Lambda)  x^\Theta \partial_k - x^\Lambda \partial_{k}(x^\Theta) \partial_j \nonumber \\
  =&
     \begin{cases}
       \partial_{j}(x^\Lambda)  x^\Theta \partial_k & \text{if \(j<k\)}, \\
       - x^\Lambda \partial_{k}(x^\Theta) \partial_j & \text{if \(j>k\)},\\
       0 & \text{otherwise.}
     \end{cases}
\end{eqnarray*} 
This operation is then extended  by bilinearity endowing \(\Lie_n\) of the structure of Lie \(D\)-ring. We   highlight    the   parallel    between   this    definition   and
Equation~\ref{parallelism}  which   defines  the  commutator   of  two
elements of $W_n$. Notice that $\Lie_n$ can be identified with a \(D\)-subring of the Witt algebra over  \(F\) (see e.g.~\cite[Chapter 2]{strade}) having $\mathfrak{B}$ as a basis.

\section{Transfinite hypercentral series of $W_n$}
In this section, we aim to provide necessary and sufficient conditions on the modules \(R_i\) in order for \(W_n\) to be transfinitely hypercentral. 

Let \(f \colon D^i \to D\), \(d = \sum_{j=1}^i c_j e_j\), and \(d_j = \sum_{s=1}^{j-1} c_s e_s\). We define
\begin{equation}\label{eq:increment}
	\Delta(d)(f) \defeq f_d-f= \sum_{j=1}^i \Delta_j(c_j)(f_{d_j}),
\end{equation}
where \(f_{c}(x) = f(x + c)\).

\begin{definition}\label{def:virtual_poly}
	For \(i \ge 0\),	the subring \(\LP_i \) of the virtual polynomials of \(D^{D^i}\)  is defined  as 
\begin{equation}
	\LP_i = \left\{
	f \in D^{D^i}\,\middle|\,
	\exists k \in \mathbb{N} \text{ s.t. }  
	(\Delta(d_1) \cdots \Delta(d_k))(f) = 0 \text{ for all } d_1, \dots, d_k \in D^i
	\right\}
\end{equation}
A virtual polynomial \(f\in \LP_i\) is said to be of class \(k\), whenever \(k\) is the minimum non-negative integer such that \((\Delta(d_1) \cdots \Delta(d_k))(f) = 0 \text{ for all } d_1, \dots, d_k \in D^i\). 
\end{definition}

\begin{remark}\label{rem:poly1}
	Let \( f \in \LP_k \) be of class \(\ell\).  
	By Lemma~\ref{lem:finite_diffs_on_Z}, if \( x, d \in D^k \), then the function \(\hat{f}_{x,d} \colon \Z \to D\) defined by  
	\[
	\hat{f}_{x,d}(m) = f(x + m d)
	\]  
	is a polynomial function in \(m\) of degree at most \(\ell - 1\), with coefficients in \(F\); that is, \(f\) behaves like a polynomial along affine integral lines.  
	We will use the same notation \(\hat{f}_{x,d}\) to denote the natural extension of this function to a map \(\hat{f}_{x,d} \colon \Q \to F\). In particular, \(\hat{f}_{x,d}\) is a polynomial function with coefficients in \(F\) such that \(\hat{f}_{x,d}(\Z) \subseteq D\).  
	
	In general, we denote by \(P_k\) the subring of \(F[x_1, \dots, x_k]\) consisting of those polynomials \(f \in F[x_1, \dots, x_k]\) such that  
	\[
	f(d_1, \dots, d_k) \in D \quad \text{for all } (d_1, \dots, d_k) \in D^k.
	\]  
	This ring is also known as the \emph{ring of numerical polynomials} in \(k\) variables. Moreover, \(P_k \subseteq \LP_k\), and in the special case \(D = \Z\), we have \(P_k = \LP_k\). We point out that this is not always the case, e.g., if \(D=\C\) the function \(z\mapsto \re(z)\), associating to a complex number its real part, is an example of virtual polynomial of class \(2\) that belongs to \(LP_1\) but not to \(P_1\).
\end{remark}

With the same notation as in Remark~\ref{rem:poly1} we have the following result.
\begin{lemma}
	Let \(f \in \LP_i\), and suppose \(x, d \in D^i\) and \(r \in \Q\) are such that \(x + r d \in D^i\). Then  
	\[
	f(x + r d) = \hat{f}_{x,d}(r).
	\]
\end{lemma}

\begin{proof}
	Write \(r = s/t\), with \(s, t \in \Z\), and set \(d' = r d\), so that \(t d' = s d\). Note that \(\hat{f}_{x, t d'} = \hat{f}_{x, s d}\). We have
	\(f(x + t d' m) = \hat{f}_{x, t d'}(m) = \hat{f}_{x, d'}(t m)\), and similarly, \(f(x + s d m) = \hat{f}_{x, s d}(m) = \hat{f}_{x, d}(s m)\).
	Thus,
	\begin{equation*}
	f(x + r d) = f(x + d') = \hat{f}_{x, d'}(1) = \hat{f}_{x, t d'}(1/t) = \hat{f}_{x, s d}(1/t) = \hat{f}_{x, d}(s/t) = \hat{f}_{x,d}(r).\qedhere
	\end{equation*}
\end{proof}
The \(F\)-vector space \(F^i\) can be regarded as a \(\Q\)-vector space.  
Given vectors \(u_1, \dots, u_k \in F^i\), we denote by \(\langle u_1, \dots, u_k \rangle\) the \(\Q\)-vector subspace of \(F^i\) spanned by \(\{u_1, \dots, u_k\}\).
\begin{lemma}
	Let \(u_1, \dots, u_k \in F^i\), and let \(x = r_1 u_1 + \dots + r_k u_k \in \langle u_1, \dots, u_k \rangle \cap D^i\), where \(r_1, \dots, r_k \in \Q\). If \(f \in \LP_i\) is of class \(\ell\), then \(f(x)\) is a polynomial function in the variables \(r_1, \dots, r_k\), with degree bounded by a function of \(\ell\).
\end{lemma}

\begin{proof}
	Let \(y = r_2 u_2 + \dots + r_k u_k\). By the previous lemma, we have that \(f(x) = \hat{f}_{y, u_1}(r_1)\) is a polynomial function in the variable \(r_1\) of the form
	\[
	\hat{f}_{y,u_1}(r_1) = a_0(y) + a_1(y) r_1 + \dots + a_m(y) r_1^m,
	\]
	where \(m \leq \ell - 1\), and each coefficient \(a_j(y)\) is independent of \(r_1\).
	
	By induction on \(k\), we may assume that \(a_0(y) = f(r_2 u_2 + \dots + r_k u_k) = f(y)\) is a polynomial function in \(r_2, \dots, r_k\), whose degree is bounded by a function of \(\ell\).
	It's easy to see that \(\Delta^t_{u_{1}}(f)(x)=\sum_{j=0}^{m-t}b_j(y)r_1^j\in \LP_i\), where \(b_0(y)=t!\, a_t(y)\). Hence, for \(t=1,\dots, m\), also \(a_t(y)\) is a polynomial function in \(r_2,\dots,r_k\) whose degree is bounded in terms of \(\ell\), as claimed.
\end{proof}
\begin{remark}
	Let \(W\) be a \(\Q\)-vector space, and let \(f \in F^W\).  
	We say that \(f\) is an \emph{fd-polynomial function} if, for every finite-dimensional \(\Q\)-vector subspace \(V = \langle u_1, \dots, u_k \rangle\) of \(W\), the map
	\[
	\bar{f} \colon \Q^k \to F, \quad \bar{f}(r_1, \dots, r_k) := f(r_1 u_1 + \dots + r_k u_k)
	\]
	is a polynomial function in the variables \(r_1, \dots, r_k\).
	
	Furthermore, \(f\) is called a \emph{uniform fd-polynomial function} if there exists a positive integer \(t\) such that \(\deg \bar{f} \le t\) for every such subspace \(V\).  
	In this case, the \emph{total degree} of \(f\) is defined to be the maximum, taken over all such subspaces \(V\), of the degrees of the associated polynomials \(\bar{f}\).
\end{remark}
\begin{corollary}\label{cor:fd-poly}
	Let \(f \in D^{D^i}\). Then \(f \in \LP_i\) and has class at most \(\ell\) if and only if \(f\) is the restriction to \(D^i\) of a uniform fd-polynomial function \(\tilde{f} \colon F^i \to F\) of total degree at most \(\ell - 1\).
\end{corollary}

\begin{proof}
	If \(f\) has class at most \(\ell\), then by the previous lemmas its natural extension \(\tilde{f}\) to \(F^i\) restricts on every finite-dimensional \(\Q\)-vector subspace to a polynomial function of degree at most \(\ell - 1\). This shows \(\tilde{f}\) is a uniform fd-polynomial function of total degree at most \(\ell - 1\).
	
	Conversely, if such a uniform fd-polynomial extension \(\tilde{f}\) exists, then the finite differences of order \(\ell\) of \(f\) vanish on \(D^i\). Hence, \(f \in \LP_i\) and has class at most \(\ell\).
\end{proof}

Let \(f \in \LP_{i-1}\). There exist non-negative integers \(k_1, \dots, k_{i-1}\) such that the ordinal number \(\omega^{i-2}k_{i-1} + \dots + \omega k_2 + k_1\) is minimal, and
\[
\left(\prod_{j=1}^{i-1}  \Delta_j(c_{j,1}) \cdots \Delta_j(c_{j,k_j}) \right)(f) = 0
\quad \text{for all } c_{j,1}, \dots, c_{j,k_j} \in D.
\]
We then let
\begin{align}
	\tdeg(f) &= \omega^{i-2}k_{i-1} + \dots + \omega k_2 + k_1, \\
	\tdeg(f\Delta_i) &= \sum_{j=1}^{n-i} \omega^{n-j} + \tdeg(f).
\end{align} 
Whenever \(f \in P_{i-1}\), this notion of \(\tdeg\) agrees with that given in Definition~\ref{def:tdeg}.

\begin{remark}\label{rem:tdeg2}
	Let \(f \in \LP_{t-1}\) with \(\tdeg(f) = \omega^{t-2}k_{t-1} + \dots + \omega k_2 + k_1\) and assume \(k_j \ne 0\). For any choice of functions \(g_i \in \LP_{j-1}\), with \(i = 1, \dots, s\le k_j\), and \(j = 1, \dots, t - 1\), by Corollary~\ref{cor:fd-poly}, the function
	\[
	F = \left(\Delta_j(g_1) \cdots \Delta_j(g_s)\right)(f)
	\]
	lies in \(\LP_{t - 1}\), with transfinite degree 
	\[
	\tdeg(F) = \omega^{t-2}h_{t-1} + \dots + \omega h_2 + h_1,
	\]
	for some integers \(h_i\) such that \(h_i = k_i\) for \(i = j+1, \dots, t - 1\), and \(h_j < k_j\). In particular, if \(s = k_j\), then \(h_j = 0\) and \(F\) does not depend on \(x_j\).
\end{remark}

\begin{definition}
	Let \(\alpha\) be a countable ordinal. We define
	\[
	\zeta^n_\alpha = \left\{ f\Delta_n \,\middle|\, f \in \LP_{n-1} \text{ and } \tdeg(f\Delta_n) < \alpha \right\}.
	\]
\end{definition}

Note that if \(\alpha\) is a limit ordinal, then \(\zeta^n_\alpha = \bigcup_{\beta < \alpha} \zeta^n_\beta\).
\begin{lemma}\label{lemmacentri}
	Let \(p \in \LP_{\ell-1}\) and \(f\Delta_n \in \zeta^n_\alpha\). Then there exists an ordinal \(\beta < \alpha\) such that
	\(
	[f\Delta_n, p\Delta_\ell] \in \zeta^n_\beta.
	\)
\end{lemma}

\begin{proof}
	Let \(\tdeg(f) = \omega^{n-2}k_{n-1} + \dots + k_1\). Then the commutator \([f\Delta_n, p\Delta_\ell]\) lies in \(\LP_{n-1}\) and either vanishes or by Remark~\ref{rem:tdeg2} it has transfinite degree strictly less than \(\tdeg(f\Delta_n)\). Indeed
	\[
	\tdeg\left( [f\Delta_n, p\Delta_\ell] \right) = \tdeg\left( \Delta_\ell(p)(f)\Delta_n \right) = \omega^{n-2}k_{n-1} + \dots + \omega^\ell k_{\ell+1} + \omega^{\ell-1}h_\ell + \dots + h_1,
	\]
	for some integers \(h_i\) with \(h_\ell < k_\ell\).
\end{proof}

\begin{remark}\label{rem:center}
	Since \(D^i\) acts faithfully on \(R_i\) by translations, a standard argument shows that the center of \(W_n\) is given by
	\[
	Z(W_n) = \{ c\Delta_n \mid c \in D \}.
	\]
\end{remark}

\begin{lemma}
	If \(\alpha < \omega^{n-1}\) is an ordinal, then \(Z_\alpha(W_n)\cap B_n \le \zeta_\alpha\).
\end{lemma}

\begin{proof}
	We argue induction on \(\alpha\), the case \(\alpha=1\) being obvious by the previous remark. Let \(f\Delta_n\in Z_\alpha(W_n)\cap B_n\) and \(c\in D\). The commutator \([f\Delta_n, c\Delta_\ell]\in Z_\beta(W_n)\cap B_n\) for some ordinal \(\beta < \alpha\). By induction \([f\Delta_n, c\Delta_\ell]\in \zeta_{\beta}\) for every choice of \(\ell \le n\), i.e., \(\tdeg(\Delta_\ell(c)(f))< \beta\) and hence \(\tdeg(f\Delta_n)< \beta +1 \le  \alpha\) as required.
\end{proof}

A direct consequence of the previous lemmas  is the following.

\begin{corollary}\label{cor:centri}
	Let \(\alpha \le \omega^{n-1}\) be an ordinal. If for every \(i = 1, \dots, n-1\), the ring \(R_i\) is contained in \(\LP_i\), then:
	\begin{enumerate}
		\item \([\zeta^n_{\alpha+1}, W_n] \subseteq \zeta^n_\alpha\) if \(\alpha\) is not a limit ordinal;
		\item \([\zeta^n_\alpha, W_n] \subseteq \zeta^n_\alpha\) if \(\alpha\) is a limit ordinal.
	\end{enumerate}
	In particular, \(\zeta^n_\alpha = Z_\alpha(W_n) \cap B_n\).
\end{corollary}

\begin{theorem}\label{thm:hyper}
	 The group \(W_n\) is transfinite hypercentral if and only if \(n=1\) or  \(R_i \subseteq \LP_i\) for all \(i = 1, \dots, n - 1\).
\end{theorem}

\begin{proof}
	We may assume \(n\ge 2\), the result being trivial for \(n=1\).
	Assume first that \(W_n\) is transfinite hypercentral. Let \(f\Delta_\ell \in Z_\alpha(W_n)\) with \(f \in R_{\ell-1}\). We proceed by transfinite induction on \(\alpha\), with the goal of showing that \(f \in \LP_{\ell-1}\).
	
	If \(\alpha = 1\), then by  Remark~\ref{rem:center} we must have \(\ell = n\) and \(f = c\) for some \(c \in D\). Thus, \(\Delta(d)(f) = 0\) for all \(d \in D^{n-1}\).
	
	Now assume \(\alpha\) is a countable non-limit ordinal. For every \(i = 1, \dots, \ell - 1\) and every \(c \in D\), we have
	\[
	[f\Delta_\ell, c\Delta_i] = \Delta_i(c)(f)\Delta_\ell \in Z_{\alpha - 1}(W_n).
	\]
	By the inductive hypothesis, there exists \(m \in \mathbb{N}\) such that
	\[
	(\Delta(d_1) \cdots \Delta(d_m) \Delta_i(c_i))(f) = 0
	\]
	for all \(d_1, \dots, d_m \in D^{\ell - 1}\) and all \(c_i \in D\).
	
	Let \(d_{m+1} = \sum_{i=1}^{\ell - 1} c_i e_i\). By Equation~\eqref{eq:increment} and the translation invariance of \(\LP_{\ell - 1}\), it follows that
	\[
	(\Delta(d_1) \cdots \Delta(d_m) \Delta(d_{m+1}))(f) = 0
	\]
	for all \(d_1, \dots, d_{m+1} \in D^{\ell - 1}\), showing that \(f \in \LP_{\ell - 1}\).
	
	If \(\alpha\) is a limit ordinal, observe that \(f\Delta_\ell \in \bigcup_{\beta < \alpha} Z_\beta(W_n)\), so there exists a non-limit ordinal \(\beta < \alpha\) such that \(f\Delta_\ell \in Z_\beta(W_n)\), and the previous argument applies.
	
	Conversely, suppose \(R_i \subseteq \LP_i\) for all \(i = 1, \dots, n - 1\). We proceed by induction on \(n\). The case \(n = 1\) is trivial. Assume the result holds for \(n - 1\). By Corollary~\ref{cor:centri}, we have \(Z_{\omega^{n - 1}}(W_n) \supseteq B_n\), so that the quotient \(W_n / Z_{\omega^{n - 1}}\) is a quotient of \(W_{n - 1} = W_n / B_n\), which is transfinite hypercentral by the inductive hypothesis. Hence, \(W_n\) is also transfinite hypercentral.
\end{proof}
\begin{theorem}\label{thm:hypercentralpolynomials}
  If \(R_i\le \LP_i\) and \(R_i\otimes_D F\supseteq F[x_1,\dots,x_i]\)   for all   \(i=1,\dots,   n-1\),   then
  \begin{equation}
  	W_n=Z_{\omega^{n-1}+\dots +\omega+1}(W_n).
  	\end{equation}
\end{theorem}
We split the proof into several Lemmas.

We remind the reader that the symbol \(\Delta_i\) is used to denote both the element \(1\Delta_i\) of the \(i\)-th base subgroup and  the difference operator with respect to the \(i\)-th variable. It will be clear from the context  which one we are using.  

From now on, we assume that \(R_i \subseteq \LP_i\) and that
\(
R_i \otimes_D F \supseteq F[x_1, \dots, x_{n-1}]
\)
for all \(i\). This is equivalent to requiring that for every monomial \(x^\Lambda \in F[x_1, \dots, x_{n-1}]\) there exists a non-zero element \(d_\Lambda\in D\) such that \(d_\Lambda x^\Lambda\in R_i\).

\begin{lemma}\label{lem:ZaphainBn}
	Let \(\alpha< \omega^{n-1}\) be an ordinal. There exist \(f \in \LP_{n-1}\) with \(\tdeg(f\Delta_n) = \alpha\) such that  for every ordinal \(\beta\) such that \(\beta + 1 < \alpha\), there exists an element \(h_\beta \in W_n\) such that
	\[
	\tdeg([f\Delta_n, h_\beta]) > \beta.
	\]
	In particular \(\zeta^n_{\alpha+1}\supsetneq \zeta^n_{\alpha}\).
\end{lemma}

\begin{proof}
	Write the ordinals \(\alpha\) and \(\beta + 1\) in their Cantor normal form:
	\[
	\alpha = \omega^{n-2}a_{n-1} + \dots +  \omega a_2 + a_1, \quad 
	\beta + 1 = \omega^{n-2} b_{n-1}+ \dots + \omega b_2 + b_1.
	\]
	Let \(f=cx^\Lambda\in R_{n-1}\) with \(0\ne c\in D\) and with \(\Lambda=(a_1,\dots,a_{n-1})\). We have
 \(\tdeg(x^\Lambda\Delta_n)=\alpha\). 
	Let \(j \in \{1, \dots, n - 1\}\) be the maximum index such that \(a_j > b_j\).
	
\noindent	
If $a_j-b_j>1$, then
\begin{equation*}
	\tdeg\left([f\Delta_n,\Delta_j]\right)>\beta.
\end{equation*}
If $a_j-b_j=1$, then
\begin{align*}
	\tdeg\bigl([f\Delta_n, c x_{j-1}^{|a_{j-1}-b_{j-1}|+1}\Delta_j]\bigr)&>\beta, \text{\ \ if \(j\ne 1\),}\\
	\tdeg\left([f\Delta_n,\Delta_1]\right)&>\beta, \text{\ \ if \(j= 1\).} \qedhere
\end{align*}
\end{proof}

\begin{lemma}
	Let \(g \in W_n \setminus B_n\). Then there exists \(h \in W_n\) such that \([g, h] \notin B_n\), unless \(g \in Z(W_{n-1})B_n\).
\end{lemma}

\begin{proof}
	Let \(g = \sum_{i=1}^{n-1} g_i \Delta_i \in W_n \setminus B_n\), and suppose \(g \notin Z(W_{n-1})B_n\). Then, the image of \(g\) in the quotient \(W_n / B_n \simeq W_{n-1}\) is not central. We can choose \(i\in \lbrace 1,\dots,n-1\rbrace\) minimal such that \(g_i\) depends nontrivially on the variable \(x_j\), for some \(j<i\).
	
	Consider the commutator
	\[
	[g, \Delta_j] = [g_i\Delta_i,\Delta_j]^{g_{i-1}\Delta_{i-1}\cdots g_1\Delta_1}\mod (B_{i+1}\cdots B_n)
	\]
	by minimality of \(i\). Hence, the term \([g,\Delta_j]\) lies in  \(B_i\) modulo \((B_{i+1}\cdots B_n)\) and therefore \([g, \Delta_j] \notin B_n\).
\end{proof}

\begin{lemma}
	Let \(0\ne c\in D\) be such that  \(c\Delta_{n-1}\in \W_{n-1}\). For every ordinal \(\alpha < \omega^{n-1}\) there exists \(h\in W_n\) such that \([c\Delta_{n-1},h]\notin \zeta^n_{\alpha}\).
\end{lemma}
\begin{proof}
	Write  \(\alpha = \omega^{n-2}a_{n-1} + \dots +  \omega a_2 + a_1 \). If \(\Lambda=(a_,\dots, a_{n-2}, a_{n-1}+2 )\), then \(h=x^\Lambda\Delta_n\) is the desired element.
\end{proof}

A direct consequence of the previous results is the following.

\begin{corollary}\label{corcor}
	\(
	Z_{\omega^{n-1} + 1}(W_n) = Z(W_{n-1})B_n
	\) and 
		\(
	Z_{\omega^{n-1}}(W_n) = B_n
	\).	
	In particular  
	\(	\zeta_\alpha = Z_\alpha(W_n) \subseteq B_n\) whenever \(\alpha \le \omega^{n-1}\).
\end{corollary}

        We are now ready to give  the claimed proof.
        \begin{proof}[Proof of Theorem~\ref{thm:hypercentralpolynomials}]
        We argue by induction on \(n\). The case \(n=1\) being trivial.

By Corollary~\ref{corcor}, we know that
          $Z_{\omega^{n-1}}(W_n)=B_n$ and by inductive hypothesis we  may
          assume
          that
          \[Z_{\omega^{n-2}+\dots+\omega+1}(W_{n-1})=W_{n-1}\cong
            W_n/B_n=W_n/(Z_{\omega^{n-1}}(W_n)).\] Hence \(W_n\) is transfinite hypercentral and \(W_n=Z_{\omega^{n-1}+\dots+\omega+1}(W_{n-1})\).
        \end{proof}
	\begin{remark}
          When \(R_i\otimes F=F[x_1,\ldots, x_i]\) we can give an explicit  description of the $\alpha$-th term
          of  the  upper  central  series  of  $W_n$  by  way  of  the
          definition of transfinite degree, more specifically
          \begin{equation}\label{centers of W_n}
            Z_\alpha(W_n)=\{g\in W_n \mid \tdeg(g) < \alpha\}
          \end{equation}
          and
          \begin{equation}\label{followingcenters}
            Z_{\alpha+1}(W_n)= \{c x^\Lambda\Delta_k \mid c\in D\} \ltimes  Z_\alpha
          \end{equation} where \(\tdeg(x^\Lambda\Delta_k)=\alpha\).
	\end{remark}
%
	
        \section{The correspondence between \(W_n\) and \(\Lie_{n}\)
          }
       
           \subsection{Saturated subgroups}
          From now on, unless explicitly stated otherwise, we consider  $R_i=  D[x_1,\dots,x_i]$,  for $i=1,\dots,n-1$. 
          In this case every element $f\in W_n$  can be uniquely
          decomposed  as the product $  f=\prod_\alpha  c_\alpha  b_\alpha$,  where
          \(c_\alpha\in  D\) and  \(b_\alpha\)  is  the unique  monomial
          element of   \(\mathcal{B}\)        such       that
          \(\tdeg b_\alpha= \alpha\).
          



          \begin{definition}
          	Let $H\le  W_n$ be a  $D$-subgroup, we shall  say that $H$  is a
          	saturated         subgroup        of         $W_n$        if
          	$f=\prod_\alpha     c_\alpha     b_\alpha\in    H$,     with
          	\(c_\alpha\in D\), then  $b_\alpha \in H$ for  all 
          	ordinals $\alpha$.
          \end{definition}
          \begin{remark}\label{remorder}
          	Notice  that for  \(H\)  to be  saturated  is 
          	sufficient that if $f\in H$ and \(\M(f)=c_\alpha b_\alpha\),
          	where  \(c_\alpha\in  D\) and  \(b_\alpha\in  \mathcal{B}\),
          	then $b_\alpha\in H$.  Another  equivalent condition is that
          	\(H = \left< c_\alpha b_\alpha \mid c_\alpha\in D \text{ and
          	} b_\alpha\in  H\cap \mathcal{B} \right>\).  An immediate
          	consequence                      is                     that
          	\(H=(H\cap B_n)\rtimes \dots \rtimes (H\cap B_1)\).
          \end{remark}

          \begin{lemma}\label{prophomo1}
          	Let $H\le W_n$ be a saturated subgroup. 
          	If $f\in \N_{W_n}(H)$ and \(\M(f)=c_\alpha b_\alpha\), then  $b_\alpha \in \N_{W_n}(H)$.
          \end{lemma}
          
          \begin{proof}
          	Let \(f=f_\alpha \prod_{\beta<\alpha} f_\beta\), where \(f_\alpha=\M(f)\)
          	and let \(x^\Lambda\Delta_u\) be a generic monic monomial element of \(H\). 
          	By hypotheses we have
          	\begin{equation*}
          		\textstyle	H\ni [f,\, x^\Lambda\Delta_u]=[f_\alpha,\, x^\Lambda\Delta_u][f_\alpha,\, x^\Lambda\Delta_u,\, \prod_{\beta<\alpha} f_\beta] [\prod_{\beta<\alpha} f_\beta, \, x^\Lambda\Delta_u].
          	\end{equation*}
          	We have to show that \([b_\alpha, x^\Lambda\Delta_u]\in H\).
          	Notice that \(f_\alpha=c_\alpha x^\Theta\Delta_k\in B_k\) for some \(k\) and for some \(\Theta\), so that we can write \([f_\alpha,\,  x^\Lambda\Delta_u]=d_1\cdots d_i\) where the \(d_j\)'s are monomial elements in decreasing order \(d_1\succ \dots \succ d_i \). 
          	
          	If \(k=u\), then \([b_\alpha,\, x^\Lambda\Delta_u]=1\in H\) and anything is left to be proved.
          	
          	If \(k<u\), then \([f_\alpha,\, x^\Lambda\Delta_u] = \sum_{r=1}^{\lambda_k}\tfrac{\partial^r x^\Lambda}{\partial x_k^r}\tfrac{(c_\alpha x^{\Theta})^r}{r!} \Delta_u \), where $\lambda_k$ is the degree of $x_k$ in $x^\Lambda$. Thus, \(d_r=\tfrac{\partial^r x^\Lambda}{\partial x_k^r}\tfrac{(c_\alpha x^{\Theta})^r}{r!} \Delta_u \). It is easy to see that \( d_1\) is the leading term of the element 
          	\([f,\, x^\Lambda\Delta_u]\in H\). Since \(H\) is a saturated  subgroup it follows that \(d_1\in H\). In the same fashion, by induction \((r+1) d_{r+1}\) is the leading term of \([d_{r},f]\in H\), and so \(d_{r+1}\in H\). We get that $d_r\in H$ for all $r$, and consequently, \([f_\alpha,\, x^\Lambda\Delta_u]\in H\). Notice that, since \(d_j=c_j b_j\) is an element of the saturated subgroup \(H\), with  \(c_j \in D\)  and \(b_j\in\mathcal{B}\), then \(b_j\in H\) for all \(1\leq j\leq i\), and so \([b_\alpha,\, x^\Lambda\Delta_u]\in H\)
          	
          	If \(k>u\), then $d_r=\sum_{r=1}^{\theta_k} c_\alpha \tfrac{ \partial^r  x^\Theta}{x_u^r}\tfrac{(x^\Lambda)^r}{r!}\Delta_k$ and the rest of the proof is analogous to the case $k<u$ arguing induction to see that \((r+1)d_{r+1} = M([d_r, x^\Lambda \Delta_u]) \in H\).
          \end{proof}

          By Remark~\ref{remorder} and Lemma~\ref{prophomo1} the following statement follows. 
          \begin{theorem}\label{omogeneità}
          	The normalizer in $W_n$ of a saturated subgroup is also saturated.
          \end{theorem} 
          \begin{definition}
          	Let $G\leq W_n$, we shall say that $G$ is full if for all $g=\prod_\alpha c_\alpha b_\alpha\in G$, with $c_\alpha=d\bar c_\alpha\in D$ for some \(d\ne 0\), also $\bar g=\prod_\alpha \bar c_\alpha b_\alpha\in G$.
          \end{definition}
          Notice that if $D$ is a field, every $D$-subgroup is full.
           \begin{remark}
          	An easy argument shows that a D-subgroup \(H\le W_n\) is saturated if and only if it is \(D\)-generated by \(H\cap \mathcal{B}\), i.e., 
          	the elements of \(H\) are exactly those that can be (uniquely in decreasing order) written in the form	\(\prod_{g\in H\cap \mathcal{B}} c_gg \), where \(c_g\ne 0\) for finitely many \(g\). In particular every saturated subgroup is full. The converse is not true, indeed the \(D\)-subgroup \(H=\Set{d(x_1+x_2)\Delta_n \mid d\in D}\) is a full subgroup that is not saturated.
          \end{remark}
          From now on  we call \emph{full \(D\)-normal closure} of a subgroup \(H\le W_n\) the minimal full normal  \(D\)-subgroup of \(W_n\) containing \(H\).
           	\begin{lemma}\label{normalclosure}
          		Let \(g\in W_n\) be an element of transfinite degree \(\alpha\). The full \(D\)-normal closure \(N\) of the \(D\)-subgroup \(H=D\left< g\right>\) is \(Z_{\alpha+1}\).
          	\end{lemma}
          	\begin{proof}
      	Let   \(cx^\Lambda\Delta_k\) be the leading term \(\M(g)\) of \(g\ne 1\), for some $1\le k\le n$. If \(\alpha=0\) then \(H=Z_1\) and  nothing needs  to be  proved. We now argue by transfinite induction on \(\alpha\).
        Suppose first that \(\Lambda=0\) so that \(\M(g)=c\Delta_k\) with \(k<n\) and \(c\in D\). The transfinite degree of  \(\M(g)\) is then \(\alpha=w^{n-1}+\dots+w^{k}\). If  \(\beta_{s}= w^{n-1}+\dots+w^{k+1}+sw^{k-1}\) then \(\sup_{s}\beta_s=\alpha\).  The commutator element \(u_s=[x_k^{s+1}\Delta_{k+1}\,,\, g]\in N\) has transfinite degree \(\beta_s< \alpha\), hence, by transfinite induction, the full \(D\)-normal closure of \(D\left< u_s\right>\) is \(Z_{\beta_s+1}\). It follows that \(N\ge \bigcup_s Z_{\beta_s+1}=Z_\alpha\).  Since \(g\equiv c\Delta_k \bmod Z_{\alpha}\) and \(N\) is full, the element \(\Delta_k\) belongs to \(N\).  Thus, \(N\ge D\left< \Delta_k\right>\rtimes Z_\alpha = Z_{\alpha+1}\ge H\) and so \(N=Z_{\alpha+1}\).
        We now deal with the case \(\Lambda\ne 0\). Let \(\ell\) be the minimum index with respect to the condition \(\lambda_\ell\ne 0\). Then $x^\Lambda=x_\ell x^\Theta$ where $\theta_i=\lambda_i-\delta_{i\ell}$. Let \(x^{\Gamma_s}=(x_{\ell-1})^s\)  if \(\ell>1\) and \(x^{\Gamma_s}=1\) otherwise. The commutator \([ g\,,\,  x^{\Gamma_s}\Delta_{\ell}] \in N\) has transfinite degree \(\epsilon_s\), with \(\sup_{s}\epsilon_s=\alpha\) when \(\ell>1\) and \(\sup_{s}\epsilon_s=\alpha-1\) when \(\ell=1\). Reasoning as above we have \(Z_\alpha\le N\) and finally \(N=Z_{\alpha+1}\). 
  \end{proof}          	
          \begin{proposition}\label{prop:saturatednormal}
          A normal full $D$-subgroup \(H\) of $W_n$ is a term \(Z_\alpha\) of the transfinite upper central series of \(W_n\). In particular \(H\) is saturated.
          \end{proposition}
          \begin{proof}
          Let \(\alpha=\sup_{h\in H}(\tdeg(h))+1\). Note that \(H\le Z_{\alpha}\). By Lemma~\ref{normalclosure} the subgroup \(H\) contains \(\bigcup_{\beta \le \alpha}Z_\beta = Z_{\alpha}\).
%
%
%
          \end{proof}
          \medbreak
          
          \subsection{The correspondence}
         
          	With notation as in Equation~\ref{eq:base algebra}, we  now define a map connecting the structures of \(W_n\) and of  \(\Lie_n\). 
          	\begin{definition}
          		Let \(\phi\colon \mathcal{B}\cup \Set{1} \to \mathfrak{B} \cup \Set{0}\) be defined by setting \(\phi(x^\Lambda\Delta_k)=x^\Lambda\partial_k\) and \(\phi(1)=0\). We extend this map to  \(W_n\) as follows. Let \(1\ne g\in W_n\), then there exists a unique ordinal \(\alpha\) such that \(g\in Z_{\alpha+1}\setminus Z_\alpha\). By  Equation~\eqref{followingcenters}, 
          		 \(g= cx^\Lambda\Delta_k \cdot h\) for unique \(cx^\Lambda\Delta_k\in Z_{\alpha+1}\setminus Z_\alpha\) and \(h\in Z_\alpha\), then we  set \(\phi(g)=cx^\Lambda\partial_k\).
          	\end{definition}

          \begin{lemma}\label{lem:phisuicommutatori}
          	For  $g,h\in D\mathcal{B}$ we have \(\phi([g,h])=[\phi(g),\phi(h)]\).
          \end{lemma}
          \begin{proof}
          	It suffices to note that if $g=cx^\Lambda\Delta_k$ and $h=dx^\Theta\Delta_\ell$ with $k>\ell$ and $c,d\in D$, then by Lemma~\ref{lem:maxmon}
          	\begin{equation*}
          		\M([cx^\Lambda\Delta_k,dx^\Theta\Delta_\ell])=cd\frac{\partial x^\Lambda}{\partial x_\ell}x^\Theta\Delta_k.
          	\end{equation*}
          	Thus, $\phi([g,h])=cd\dfrac{\partial x^\Lambda}{\partial x_\ell}x^\Theta\partial_k=[cx^\Lambda\partial_k,dx^\Theta\partial_\ell]$.
          \end{proof}

          We recall the definition of \emph{homogeneous}\fxnote{uno solo dei due homogeneous in corsivo} subring of $\Lie_n$ introduced in\cite{char0algebra}.
          \begin{definition}
          	A  Lie  subring  \(\mathfrak{H}\)  of  \(\Lie_{n}\)  is  said  to  be
          	homogeneous if it  is the free \(D\)-module  spanned by some
          	subset of \(\mathfrak{B}\).
          \end{definition}
          Arguing as in the proof of Theorem~\ref{omogeneità}, we can prove the following result, also proved in \cite[Theorem~1.2]{char0algebra}.
          \begin{theorem}\label{omogeneitàlie}
          	The idealizer in $\Lie_n$ of a homogeneous subring is also homogeneous.
          \end{theorem}
                Let \(H\le W_n\), we 
          denote by \(H^\phi\) the  \(D\)-subring  of \(\Lie_{n}\) generated by \(\phi(H)\).            \begin{remark}
          	If \(H\le W_n\) is a saturated subgroup, then \(H^\phi\) is a homogeneous subring of \(\Lie_{n}\). 
          \end{remark}
          
          The following statement is a trivial consequence of  Lemma~\ref{lem:phisuicommutatori} and the previous remark. 
          \begin{lemma}
          The $\alpha$-th center \(\mathfrak{Z}_\alpha\) of \(\Lie_n\)is an homogeneous subring of $\Lie_n$. Moreover \[\mathfrak{Z}_\alpha=\left<cx^\Lambda\partial_k\mid c\in D\text{ \textrm{and} } \tdeg(x^\Lambda\partial_k)<\alpha \right>=Z_{\alpha}^\phi.\]
          In particular \(\Lie_n\) is a transfinite hypercentral Lie ring over \(D\).
          \end{lemma}
            
           Noting that the map \(\phi\) yields a bijection between the set of saturated  normal subgroups of \(W_n\) and the set of homogeneous Lie ideals of \(\Lie_n\), we have the following completely analogous result to Proposition~\ref{prop:saturatednormal}.  
           \begin{proposition}
           	Every homogeneous ideal of \(\Lie_n\) is a term \(\mathfrak{Z}_{\alpha}\) of the transfinite upper central series of \(\Lie_n\). 
           \end{proposition}
            In general the map \(\phi\) transforms normalizers in idealizers when restricted to saturated subgroups. The proof is a straightforward consequence of Theorem~\ref{omogeneità}, Theorem~\ref{omogeneitàlie} and Lemma~\ref{lem:phisuicommutatori}. 
          \begin{proposition}\label{prop:normalizeridealizer}
          	 Let \(H\le W_n\) be a saturated subgroup of \(W_n\). An element \(g \in W_n\) lies in the normalizer \(N_{W_n}(H)\) of \(H\) in \(W_n\) if and only if \( \phi(g)\) lies in the idealizer \(I_{\Lie_{n}}(H^\phi)\) of \(H^\phi\) in \(\Lie_{n}\).
          \end{proposition}
%
          
          \begin{corollary}\label{cor:ranks}
          	The difference \(|N_{W_n}(H)\cap \mathcal{B}| - |H\cap \mathcal{B}|\) is equal to the rank of \(I_{\Lie_{n}}(H^\phi)/H^\phi\) as a free $D$-module.
          \end{corollary}

          \section{A sequence of normalizers}
Let $T=\langle \Delta_1,\dots,\Delta_n\rangle$  be the regular canonical  abelian subgroups  of $W_n$  generated by  the unit constant function  elements. We deal with the normalizer chain in $W_n$ starting from $T$, in other words with the sequence $\Set{\mathbf{N}_i}_{i\geq -1}$ defined as follows
        \begin{equation*} \mathbf{N}_i=
        \begin{cases}
        	T & i=-1,\\
        	N_{W_n}(T) & i=0,\\
        	N_{W_n}(\mathbf{N}_{i-1}) & i \geq 1.
        \end{cases}
        \end{equation*}

By way of Proposition~\ref{prop:normalizeridealizer}, the previous normalizer chain can be described using  the  results obtained  in \cite{char0algebra} for the case of idealizers. For the convinience of the reader, we give a brief summary.

 \begin{definition}\label{defri}
	Let $i\geq 1$ be an integer  and let \(1\le r_i \le n-1\) be
	such that \(i\equiv r_i \bmod (n-1)\).  We set 
	$$h_i\defeq  \biggr\lfloor \frac{i-1}{n-1}\biggl\rfloor  +1.$$ 
	The functions  \emph{weight-degree} and
	\emph{$i$-th level-function} on  power monomials elements are defined as 
	\begin{equation*}
		\wdd(cx^\Lambda\Delta_k)=\wt(\Lambda)-\deg(x^\Lambda)+n-k,
	\end{equation*}
	and
	\begin{equation*}
		\lev_i(cx^\Lambda\Delta_k)=h_i\wdd(x^\Lambda\Delta_k)+\deg(x^\Lambda)-1.
	\end{equation*}
\end{definition}
	Let $i\geq -1$, we define the set
	$$\mathcal{N}_{i}=\left\{x^\Lambda\Delta_k\mid\ \lev_j(x^\Lambda\Delta_k)=j\ \text{for some } j\le i\right\}$$
	and we call $\mathcal{L}_i=\mathcal{N}_{i}\setminus \mathcal{N}_{i-1}.$
	
	The following result is \cite[Theorem~3.4]{char0algebra}, read via \(\phi^{-1}\).
\begin{theorem}
The   first  normalizer   $\mathbf{N}_{W_n}(T)$  of   $T$
coincides with $\langle \mathcal{N}_0\rangle$, and in general, \[\mathbf{N}_i=N_{W_n}(\langle \mathcal{N}_{i-1}\rangle )=\langle \mathcal{N}_i\rangle.\]
\end{theorem}

The map \(\phi\) allows us to translate even more results of~\cite{char0algebra} to the group \(W_n\).
 Let
\(\{a_i\}_{i=0}^\infty\)  denote   the  sequence  of  the   number  of
partitions  of  \(i\). Let  \(b_i=\sum_{j=0}^ia_j\)  be
\(i\)-th partial  sum of $\{a_i\}$ and  \(c_i=\sum_{j=0}^ib_j\) be the
\(i\)-th  partial  sum  of  $\{b_i\}$.  The  initial  terms  of  these
sequences      can      be      found      in      \cite{OEIS}      at
\href{https://oeis.org/A085360}{A085360} .				
Beyond a threshold value that
depends  quadratically on  $n$,  the sequence  $\{|\mathcal{L}_{i}|\}$
exhibits  periodic behavior and it is related to the
sequence   $\{c_i\}$.   For   a   proof  of   the   following   result
see~\cite[Corollary 2.12]{char0algebra}.
\begin{proposition}\label{prop:ranks}
  If \(i > (n-4)(n-1)\) and \(1\le k \le n\), then
  \begin{align*}
    |\mathcal{L}_{i}\cap \mathcal{B}_k|&=b_{r_i+k-n-1}\\
    \intertext{and}
    |\mathcal{L}_{i}|&= c_{r_i-1},
  \end{align*}
  where \(\mathcal{B}_k=\mathcal{B}\cap B_k\) and the value $r_i$ is as in Definition~\ref{defri}.
\end{proposition}
As a consequence, it follows that the sequence $\{|\mathcal{L}_{i}|\}$
is ultimately periodic, meaning there  exist integers $k$ and $j$ such
that $|\mathcal{L}_i| =  |\mathcal{L}_{i+k}|$ for all $i  \geq j$. For
full details, refer to~\cite[Subsection~2.4]{char0algebra}.
\subsection{The normalizer chain in a more general setting}
Proposition~\ref{prop:ranks} is stated  under the assumption
\(R_i=D[x_1,\dots,x_i]\). We now prove that this result is valid in a more  general setting.

				
We set
\[\overbar{W}_n = D\wr_{\bar R_{n-1}}D\wr_{\bar R_{n-2}} \dots
  \wr_{\bar  R_1}D  ,\] where  $\overbar  R_i$  is  a subring  of  the
numerical            polynomials           $P_i$            satisfying
$\bar R_i\otimes_D F=F[x_1,\dots,x_i]$, and we set
\[\widetilde{W}_n=
  F\wr_{\widetilde{R}_{n-1}}F\wr_{\widetilde{R}_{n-2}}           \dots
  \wr_{\widetilde{R}_1}F,\]                                      where
\(\widetilde{R}_i=F[x_1,\dots,x_i]\).                          Clearly
\(W_n \le \overbar  W_n\le \widetilde{W}_n\).  We will  also denote by
\(B_i\), \(\overbar  B_i\), and \(\widetilde{B}_i\) the  \(i\)-th base
subgroup  of   \(W_n\),  \(\overbar  W_n\)   and  \(\widetilde{W}_n\),
respectively.
				
We now consider  the normalizer chain \(\{\overbar{\mathbf{N}}_i\}\)
arising    from    \(T\)    in   \(\overbar    W_n\),    defined    by
\(\overbar{\mathbf{N}}_{-1}=T\)       and       recursively       by
$\overbar{\mathbf{N}}_i=\N_{\overbar
  W_n}(\overbar{\mathbf{N}}_{i-1})$.   Analogously,  we  define  the
normalizer      chain       \(\{\widetilde{\mathbf{N}}_i\}\)      in
$\widetilde{W}_n$.
				Notice that
\begin{equation}\label{ciao}
  (\overbar{\mathbf{N}}_{i}\cap \overbar B_k)\otimes_D F \cong \widetilde{\mathbf{N}}_i\cap \widetilde{B}_k\cong (\mathbf{N}_i\cap B_k)\otimes_D F
\end{equation}
for all \(i\geq -1\) and $1\le k\le n$. Thus,
defining  the  free  rank  \(\frk(M)\)  of  a  \(D\)-module  \(M\)  as
\(\dim_F(M\otimes_D         F)\),        we         obtain        that
\(\frk(\overbar{\mathbf{N}}_i\cap       \overbar        B_k)       =
\frk({\mathbf{N}}_i \cap B_k)\).

From Equation~\ref{ciao} the following   generalization   of
Proposition~\ref{prop:ranks} follows.
\begin{theorem}
  The         free         rank        of         the         quotient
  \((\overbar{\mathbf{N}}_i\cap                             \overbar
  B_k)/(\overbar{\mathbf{N}}_{i-1} \cap \overbar  B_k)\) is equal to
  the          free         rank          \(b_{r_i+k-n-1}\)         of
  \( (\mathbf{N}_i \cap B_k)/({\mathbf{N}  }_{i-1}\cap B_k) \) for
  all \(i>(n-4)(n-1)\). In particular it does not depend on the choice
  of the \(R_j\)'s.
\end{theorem}

\section{Free     regular     abelian      normal     subgroups     of
  \(\mathbf{N}_0\)}\label{sec:regular}
In this  section we  find an  analog of  Theorem~6 and  Corollary~3 of
\cite{regular}, where the authors prove that a Sylow \(2\)-subgroup of
\(\AGL(2^n)\) contains  exactly one normal elementary  regular abelian
subgroup distinct from the translation group and conjugated to it.
\begin{lemma}\label{center}
  If  $T<W_n$ is  a regular  abelian  group, then  $T$ intersects  the
  center $Z$ of $W_n$ non-trivially.
\end{lemma}
\begin{proof}
  It is enough to notice that $TZ$ is abelian and transitive, so it is
  regular (and contains $T$).  This means that $T=TZ$.
\end{proof}
Let  $T=\langle  \Delta_1,\ldots,\Delta_n\rangle$   be  the  canonical
regular subgroup considered in the previous section.
\begin{remark}
  Notice that \(\mathbf{N}_0=S\ltimes T\),  where \(S\) is the group
  generated by  \(\{ x_j\Delta_k \mid  1\le j<  k \le n\}\)  acting on
  \(T\)  as the  group  of upper  unitriangular  matrices. Indeed,  if
  $k>i$, then
  \begin{equation*}
    [\Delta_i,x_j\Delta_k]=\begin{cases}
      \Delta_k&\text{if } i=j\\
      1&\text{otherwise}\end{cases}
  \end{equation*}
  Moreover, $\mathbf{N}_0$ acts on $T$ by conjugation, i.e.
  $$(\Delta_i)^{x_j\Delta_k}=\Delta_i\left[\Delta_i,x_j\Delta_k\right]=\begin{cases}
    \Delta_i\Delta_k&\text{if } i=j\\
    \Delta_i&\text{otherwise}\end{cases}$$                         and
  $(\Delta_i)^{\Delta_k}=\Delta_i$. We  denote by $E_{ij}$  the elementary matrix
  with $1$  in the position $(i,j)$  and $0$ elsewhere,   by $e_i$
   the  \(i\)-th vector of the canonical basis, and by \(\mathbb{1}_n\) the \(n\times n\) identity matrix.
  The vector $e_i$ represents the element $\Delta_i$ of $T$, while the
  matrix  $\mathbb{1}_n+\delta_{ji}E_{jk}$  represents the  action  of
  $x_j\Delta_k$ on $\Delta_i$.  In other words
  $$\left(\Delta_i\right)^{x_i\Delta_k}=e_i\cdot
  \left(\mathbb{1}_n+E_{ik}\right)=e_i+e_k=\Delta_i\Delta_k.$$      By
  varying $1<i<k<n$,  these matrices generate the  upper unitriangular
  group   $U$.   Thus,   we   establish   a  surjective   homomorphism
  $\pi\colon\mathbf{N}_0\to  U$ such  that $\ker(\pi)=T$.   Moreover
  the  map  $\tau\colon  U\to   \mathbf{N}_0$,  defined  by  sending
  $E_{jk}\to    x_j\Delta_k$,    is    a    homomorphism    such that
  $\tau\pi$ is the identity.           It           follows          that
  $\mathbf{N}_0\cong U\ltimes T\cong S\ltimes T$.
\end{remark}

Let \(\T\)  be a regular  abelian subgroup normal  in $\mathbf{N}_0$
isomorphic           to           \(D^n\).            We           set
$\T=\langle           \bar{z}_1,\ldots,\bar{z}_n\rangle$          with
$\bar{z}_i= t_i v_i$ where $t_i\in T$ and $v_i\in S$.
\begin{lemma}
  Let  $t_i=\sum_{j=1}^n a_{ij}\Delta_j$,  where  $a_{ij}\in D$.   The
  matrix $(a_{ij})_{i,j=1,\ldots,n}$ is unimodular.
\end{lemma}
\begin{proof}
  The group  $\T$ is  regular, in  particular for  each $j=1,\ldots,n$
  there  exists  a  unique  permutation  $\sigma_j\in  \T$  such  that
  $\sigma_j(0)=e_j$.   Since $\bar{z}_i$'s  generate  $\T$, there  are
  coefficients         $h_{ij}\in         D$         such         that
  $\sigma_j=\sum_{i=1}^n h_{ji}\bar{z}_i$ and as a consequence
  \begin{align*}
    e_j=\sigma_j(0)=\sum_{i=1}^n h_{ji}\bar{z}_i(0)&=\sum_{i=1}^n h_{ji} t_i v_i\left(0\right).
  \end{align*}
  Notice that $v_i(0)=0$ since $v_i\in S$ and $\Delta_s(0)=e_s$, thus
  \begin{align*}
    e_j&=\sum_{i=1}^n h_{ji} t_iv_i\left(0\right)=\sum_{i=1}^n h_{ji} t_i\left(0\right)\\
       &=\sum_{i=1}^n h_{ji}\sum_{s=1}^n a_{is}\Delta_s\left(0\right)=\sum_{i=1}^n h_{ji}\sum_{s=1}^n a_{is}e_s
  \end{align*}
  Hence   $   \sum_{i=1}^n   h_{ji}a_{is}=\delta_{js}   $   for   each
  $j=1,\ldots,n$.
\end{proof}
Let us consider the inverse matrix $(b_{ij})$ of $(a_{ij})$ and let us
define $\bar{t}_i=\sum_{j=1}^n b_{ij}\bar{z}_j$. Notice that
\begin{equation*}
  \bar{t}_i\left(0\right)=\sum_{j=1}^n b_{ij}\bar{z}_j\left(0\right)
  =\sum_{j=1}^n b_{ij}t_j\left(0\right)
  =\Delta_i\left(0\right)=e_i.
\end{equation*}
Since  $T$  is regular, the only  element of $T$ sending  $0$ to $e_j$ is
$\Delta_j$,         so        that         we        can         write
$\T=\langle \bar{t}_1,\ldots,\bar{t}_n\rangle$ where
\begin{equation}\nonumber
  \bar{t}_i=\Delta_i u_i\text{ with }u_i\in S.
\end{equation}
\begin{lemma}\label{delta2}
  If $\Delta_j\in \T$, then $\Delta_{j+1},\dots, \Delta_n\in \T$.
\end{lemma}
\begin{proof}
  It is enough to  notice that $\Delta_i=[x_j\Delta_i,\Delta_j]\in \T$
  for $i=j+1,\dots,n$.
\end{proof}
We      are     now      ready     to      describe     the      group
$\T=\langle \bar{t}_1,\dots,\bar{t}_n\rangle$.
\begin{proposition}
  If $\T$  is a  regular abelian  normal subgroup  of $\mathbf{N}_0$
  isomorphic to \(D^n\), then there exists $c\in D$ such that
  \begin{equation*}
    \T=\langle \Delta_1(c x_1\Delta_n),\Delta_2,\dots,\Delta_n\rangle.
  \end{equation*}
\end{proposition}
\begin{proof}
  By  Lemma~\ref{center},  $\Delta_{n}\in  \T$  and so  $u_n$  is  the
  identity          1         of          $\T$.          Analogously
  $\Delta_{i}\in   \T\text{   mod   }(B_n\cdots   B_{i+1})$   and   so
  $u_{i}\in B_n\cdots B_{i+1}$. Hence
  \begin{equation}\label{eccola}
    \bar t_i=\Delta_i\prod_{{i+1\le k\le n}\atop{j<k}}(x_j\Delta_k)^{c_{ijk}}.
  \end{equation}
  In                                                        particular,
  \([\bar                       t_i,                       x_1\Delta_s
  ]=[\Delta_i,x_1\Delta_s]^{u_i}[u_i,x_1\Delta_s]
  =\Delta_s^{\delta_{i1}u_i^{-1}}[u_i,x_1\Delta_s]\in   \T   \).    If
  \(i\ne  1  \) we  have  \([u_i,x_1\Delta_s]\in  \T\cap S\)  so  that
  \([u_i,x_1\Delta_s]=1\).          This          implies         that
  \(u_i=\prod (x_1\Delta_{k})^{c_{i1k}}\) for \(i>1\) and
  \begin{equation*}
    \bar t_i=\Delta_i\prod_{i+1\le k\le n}(x_1\Delta_k)^{c_{i1k}} \text{ for } 1<i<n.
  \end{equation*}
  Let $i\ge 2$ and     $i+1\leq      k<n$,     then     the     commutator
  $[\bar{t}_i,x_k\Delta_n]=[\Delta_i,x_k\Delta_n]^{u_i}[u_i,x_k\Delta_n]=(x_1\Delta_n)^{-c_{i1k}}\in
  \T\cap  S$.  Thus, \((x_1\Delta_n)^{-c_{i1k}}=1\)  and  so
  $c_{i1k}=0$. It means that
  \begin{equation*}
    u_i=(x_1\Delta_n)^{c_{i1n}}\text{ for }1<i<n-1.
  \end{equation*}
  For         $1<i\leq         n-2$        we         get         that
  $[\bar{t}_i,x_i\Delta_{i+1}]=[\Delta_i,x_i\Delta_{i+1}]^{u_i}[u_i,x_i\Delta_{i+1}]=\Delta_{i+1}^{-1}[u_i,x_i\Delta_{i+1}]=\Delta_{i+1}^{-1}\in
  \T$.   In  particular  $\Delta_3\in \T$  and,  by  Lemma~\ref{delta2}, we have
  $\Delta_4,\dots,\Delta_n\in \T$. Hence, we obtain
  \begin{align}\label{eccola2}
    \bar{t}_i&=\Delta_i \text{ for }3\leq i\leq n,\nonumber\\
    \bar{t}_2&=\Delta_2(x_1\Delta_n)^{c_{21n}}
  \end{align}
  and $\bar{t}_1$ as in Equation~\eqref{eccola}.

  When             $1<             k<n$             we             get
  $[\bar{t}_1,x_k\Delta_n]=[\Delta_1,x_k\Delta_n]^{u_1}[u_1,x_k\Delta_n]=\sum_{j<
    k\le  n-1}  (x_j\Delta_n)^{c_{1jk}}\in  \T\cap S$,  hence  by
  regularity $c_{1jk}=0$ and
$$\bar{t}_1=\Delta_1(x_1\Delta_n)^{c_{11n}}\dots (x_{n-1}\Delta_n)^{c_{1,n-1,n}}.$$
					
If      $3\leq      s\leq       n$,      then      the      commutator
$[\bar{t}_1,x_2\Delta_s]=x_2\Delta_n^{c_{1sn}}\in    \T\cap   S$.
Thus     it     has     to      be     the     identity     and     so
$ c_{1sn}=0\text{ for }3\leq s< n$, i.e.
$$u_1=\prod_{{j=1,2}}(x_j\Delta_n)^{c_{1jn}}=(c_{11n}x_1+c_{12n}x_2)\Delta_n.$$

Since $\T$ is abelian, on the one hand we have
\begin{align*}
  1=[\bar{t}_1,\bar{t}_2]&=[\Delta_1u_1,\Delta_2u_2]\\
                         &=[\Delta_1,u_2]^{u_1}[\Delta_1,\Delta_2]^{u_1u_2}[u_1,u_2][u_1,\Delta_2]^{\Delta_2}\\
                         &=(\Delta_n)^{c_{12n}-c_{21n}}[u_1,u_2]
\end{align*}
and so $[u_1,u_2]=1$  and $\Delta_n^{c_{12n}-c_{21n}}=1$. In
particular                          \begin{equation}\label{eq:alphaeq}
  c_{12n}=c_{21n}.
\end{equation}
					
On the other hand, the commutator
\begin{align*}
  [\bar{t}_1,x_1\Delta_2]&=[\Delta_1,x_1\Delta_2]^{u_1}[u_1,x_1\Delta_2]\\
                         &=\Delta_2^{-u_1}(x_1\Delta_n)^{c_{12n}}\\
                         &=\Delta_2^{-1}\Delta_n^{c_{12n}}(x_1\Delta_n)^{c_{12n}}\in \T.
\end{align*}
Since        $\Delta_n\in        \T$,        we        get        that
$\Delta_2(x_1\Delta_n)^{-c_{12n}}\in       \T$       and       by
Equation~\eqref{eccola2}       $c_{21n}=-c_{12n}$.        By
\eqref{eq:alphaeq} it follows that $c_{21n}=c_{12n}=0$, i.e.
$\bar{t}_2=\Delta_2$                                               and
$\bar{t}_1=\Delta_1(x_1\Delta_n)^{c_{11n}}$.
\end{proof}
Let
\(\T_{c}\defeq               \langle              \Delta_1(c
x_1\Delta_n),\Delta_2,\dots,\Delta_n\rangle\)  be as  in the  previous
proposition.  We have the following result.
\begin{corollary}
  If  $2$   is  an   invertible  element  of   $D$,  then   for  every
  $g\in  \mathbf{N}_1$  there  exists \(c_g\in  D\)  such  that
  $T^g=\T_{c_g}$.  Otherwise in  \(\mathbf{N}_1\) there are two
  distinct conjugacy  classes of  regular abelian normal  subgroups of
  $\mathbf{N}_0$ isomorphic to \(D^n\).
\end{corollary}
\begin{proof}
  It         is         enough          to         notice         that
  $\mathcal{N}_1\setminus\mathcal{N}_0=\{x_1^2\Delta_n\}$          and
  $\Delta_i^{x_1^2\Delta_n}=\Delta_i$ for $i=2,\dots,n$. Moreover, for
  $d\in D$
  \begin{equation*}
    \Delta_1^{d x_1^2\Delta_n}=\Delta_1[\Delta_1,d x_1^2\Delta_n]=\Delta_1(2d x_1\Delta_n)
  \end{equation*}
  and
  \begin{align*}
    (\Delta_1(x_1\Delta_n))^{d x_1^2\Delta_n}&=\Delta_1(x_1\Delta_n)[\Delta_1(x_1\Delta_n),d x_1^2\Delta_n]\\
                                                  &=\Delta_1(x_1\Delta_n)[\Delta_1,d x_1^2\Delta_n]
                                                    =\Delta_1((2d+1) x_1\Delta_n)
  \end{align*}
  Note that \(\Delta_1\)  appears as a member of the  second family if
  and only if \(2\) is invertible in \(D\).
\end{proof}

\bibliographystyle{abbrv} \bibliography{citation}

\end{document}